\documentclass[oneside,english]{amsart}
\usepackage[T1]{fontenc}
\usepackage[latin9]{inputenc}
\usepackage{amsthm}
\usepackage{amstext}
\usepackage{amssymb}
\usepackage{esint}
\usepackage{graphicx}
\usepackage{enumerate}

\usepackage{graphicx}
\graphicspath{{noiseimages/}}

\makeatletter
\newtheorem{theorem}{Theorem}[section]
\newtheorem{lemma}[theorem]{Lemma}

\numberwithin{figure}{section}
\theoremstyle{definition}

\theoremstyle{remark}
\newtheorem{remark}[theorem]{Remark}

\numberwithin{equation}{section}
\makeatother

\usepackage{babel}

	\DeclareMathOperator{\loc}{loc}
	
	\DeclareMathOperator*{\esssup}{ess\,sup}

\begin{document}

\title[Spectral Properties]{Spectral Estimates of the $p$-Laplace Neumann operator and Brennan's Conjecture}

\author{V.~Gol'dshtein, V.~Pchelintsev, A.~Ukhlov}

\begin{abstract}
In this paper we obtain estimates for the first nontrivial eigenvalue of the $p$-Laplace Neumann operator in bounded simply connected planar domains $\Omega\subset\mathbb R^2$. This study is based on a quasiconformal version of the universal weighted Poincar\'e-Sobolev inequalities obtained in our previous papers for conformal weights. The suggested weights in the present paper are Jacobians of quasiconformal mappings. The main technical tool is the theory of composition operators in relation with the Brennan's Conjecture for (quasi)conformal mappings. 
\end{abstract}
\maketitle
\footnotetext{\textbf{Key words and phrases:} elliptic equations, Sobolev spaces, quasiconformal mappings.} 
\footnotetext{\textbf{2010
Mathematics Subject Classification:} 35P15, 46E35, 30C65.}

\section{Introduction}

In this paper we obtain lower estimates for the first nontrivial eigenvalue of the $p$-Laplace Neumann operator with the Neumann boundary condition
\[
\begin{cases}
-\textrm{div}(|\nabla u|^{p-2}\nabla u)=\mu_p|u|^{p-2}u & \text{in $\Omega$}\\
\frac{\partial u}{\partial n}=0 & \text{on $\partial \Omega$},
\end{cases} 
\]
in a bounded simply connected planar domain $\Omega \subset \mathbb{R}^2$.
The weak statement of this spectral problem is as follows: a function $u$ solves the previous problem iff 
$u \in W^{1}_{p}(\Omega)$ and
$$
\int\limits _{\Omega} (|\nabla u(x)|^{p-2}\nabla u(x) \cdot \nabla v(x))\,dx =
\mu_p \int\limits _{\Omega} |u(x)|^{p-2} u(x) v(x)\,dx
$$ 
for all $v \in W^{1}_{p}(\Omega)$. 

We demonstrate that integrability of Jacobians of quasiconformal mappings permit us to obtain lower estimates of the first non-trivial eigenvalue $\mu_{p}^{(1)}(\Omega)$ in terms of Sobolev norms of a quasiconformal mapping of the unit disc $\mathbb{D}$ onto $\Omega$. So, we can conclude that $\mu_{p}^{(1)}(\Omega)$ depends on the quasiconformal geometry of $\Omega$:

\vskip 0.3cm
{\bf Theorem A.}
\textit{Let $\Omega\subset\mathbb R^2$ be a $K$-quasiconformal $\alpha$-regular domain and $\varphi : \Omega \to \mathbb D$ be a $K$-quasiconformal mapping. Suppose that the Brennan's Conjecture holds.  
Then for every
$$
p \in \left(\max \left\{\frac{4K}{2K+1},\frac{\alpha (2K-1)+2}{\alpha K}\right\} \,, 2\right)
$$ 
the following estimate
$$
\frac{1}{\mu_p^{(1)}(\Omega)} \leq K \|J_{\varphi^{-1}}\mid L_{\frac{\alpha}{2}}(\mathbb D)\| \inf_{q \in I}
\left\{
B^p_{\frac{\alpha p}{\alpha -2},q}(\mathbb D) 
\||D\varphi^{-1}|^{p-2}\mid L_{\frac{q}{p-q}}(\mathbb D)\| 
\right\}
$$
holds. }

Here $I=[1,2p/(4K-(2K-1)p))$ and $B_{r,q}(\mathbb D)$ is the best constant in 
the corresponding $(r,q)$-Poincar\'e-Sobolev inequality in the unit disc $\mathbb D$ for $r=\alpha p/(\alpha -2)$.

\vskip 0.2cm

Let us give few detailed comments to the theorem:
\vskip 0.2cm

\noindent
{\sc 1.1 $K$-quasiconformal $\alpha$-regular domains.}
Recall that a homeomorphism $\varphi:\Omega \rightarrow \widetilde{\Omega}$
between planar domains is called $K$-quasiconformal if it preserves
orientation, belongs to the Sobolev class $W_{2,\loc}^{1}(\Omega)$
and its directional derivatives $\partial_{\alpha}$ satisfy the distortion inequality
$$
\max\limits_{\alpha}|\partial_{\alpha}\varphi|\leq K\min_{{\alpha}}|\partial_{\alpha}\varphi|\,\,\,
\text{a.e. in}\,\,\, \Omega \,.
$$

The notion of conformal regular domains was introduced in~\cite{BGU} and was used for study conformal spectral stability of the Laplace operator. In the present work we introduce the more general notion of quasiconformal regular domains.

A simply connected planar domain $\Omega \subset \mathbb{R}^2$ is called a $K$-quasiconformal $\alpha$-regular domain if there exists a $K$-quasiconformal mapping $\varphi : \Omega \to \mathbb D$ such that
$$
\int\limits_\mathbb D |J(y, \varphi^{-1})|^{\frac{\alpha}{2}}~dy < \infty \quad\text{for some}\quad \alpha >2.
$$
The domain $\Omega \subset \mathbb{R}^2$ is called a $K$-quasiconformal regular domain if it is a $K$-quasiconformal $\alpha$-regular domain for some $\alpha>2$.

Note that class of quasiconformal regular domain includes the class of Gehring domains \cite{AK} and can be described in terms of quasihyperbolic geometry \cite{KOT}.

\begin{remark}
The notion of quasiconformal $\alpha$-regular domain is more general then the notion of conformal $\alpha$-regular domain. Consider, for example, the unit square $\mathbb Q\subset\mathbb R^2$. Then $\mathbb Q$ is a conformal $\alpha$-regular domain for $2<\alpha\leq 4$ \cite{GU16} and is a quasiconformal $\alpha$-regular domain for all $2<\alpha\leq \infty$ because the unit square $\mathbb Q$ is quasiisometrically equivalent to the unit disc $\mathbb D$.
\end{remark}

\begin{remark}
Because $\varphi: \Omega\to\mathbb D$ is a quasiconformal mapping, then integrability of the derivative is equivalent to integrability of the Jacobian:
$$
\int\limits_\mathbb D |J(y,\varphi^{-1})|^{\frac{\alpha}{2}}~dy\leq \int\limits_\mathbb D |D \varphi^{-1}(y)|^{\alpha}~dy 
\leq K^{\frac{\alpha}{2}}\int\limits_\mathbb D |J(y,\varphi^{-1})|^{\frac{\alpha}{2}}~dy.
$$
\end{remark}

\vskip 0.2cm

{\sc 1.2 Brennan's Conjecture.}
We can conclude from the theorem A that Brennan's Conjecture leads to the spectral estimates of the $p$-Laplace operator in quasiconformal $\alpha$-regular domains $\Omega\subset\mathbb R^2$.

\textbf{Generalized Brennan's Conjecture} for quasiconformal mappings \cite{HSS} states that
\begin{equation}
\int\limits _{\Omega}|D\varphi(x)|^{\beta}~dx<+\infty,\quad\text{for all}\quad\frac{4K}{2K+1}<\beta<\frac{2K\beta_0}{(K-1)\beta_0+2}.\label{eq:HSS1}
\end{equation}

If $K=1$ we have Brennan's conjecture for conformal mappings \cite{Br} which is proved for $\beta\in (4/3,\beta_0)$, where $\beta_0=3.752$ \cite{HShim}.

\vskip 0.2cm

{\sc 1.3 Historical sketch.}
In 1961 G.Polya  \cite{P60} obtained upper estimates for eigenvalues of Neumann-Laplace operator in so-called plane-covering domains. Namely, for the first eigenvalue:
$$
{\mu_2^{(1)}(\Omega)}\leq 4\pi|\Omega|^{-1}.
$$

The lower estimates for the $\mu_p^{(1)}(\Omega)$ were known before only
for convex domains. In the classical work \cite{PW} it was proved that if $\Omega$ is convex with diameter $d(\Omega)$ (see, also \cite{ENT,FNT,V12}), then
\begin{equation}
\label{eq:PW}
\mu_2^{(1)}(\Omega)\geq \frac{\pi^2}{d(\Omega)^2}.
\end{equation}

In \cite{ENT} was proved that
if $\Omega \subset \mathbb R^n$ is a bounded convex domain having diameter
$d$ then for $p \geq 2$ 
$$
\mu_p^{(1)}(\Omega) \geq \left(\frac{\pi_p}{d(\Omega)}\right)^p
$$
where
$$
\pi_p = 2 \int\limits_0^{(p-1)^{\frac{1}{p}}} \frac{dt}{(1-t^p/(p-1))^{\frac{1}{p}}} =
2 \pi \frac{(p-1)^{\frac{1}{p}}}{p \sin(\pi/p)}.
$$ 

In the case of non-convex domains in \cite{GU2016} was proved that if $\Omega \subset \mathbb R^2$ be a conformal $\alpha$-regular domain
then for every $p \in (\max \{4/3,(\alpha + 2)/\alpha\} \,, 2)$ the following inequality holds 
\begin{multline*}
\frac{1}{\mu_p^{(1)}(\Omega)} \leq \\
{} \inf_{q \in [1,2p/(4-p))} 
\left\{
B_{\frac{\alpha p}{\alpha -2},q}^p(\mathbb D)
\cdot \|D\varphi^{-1}|^{p-2}\,|\,L_{\frac{q}{p-q}}(\mathbb D)\| 
\right\}\|D\varphi^{-1}|\,|\,L_{\alpha}(\mathbb D)\|^2.
\end{multline*} 

\vskip 0.3cm

The first non-trivial eigenvalue of the Neumann boundary problem for the $p$-Laplace operator $\mu_p^{(1)}(\Omega)^{-\frac{1}{p}}$ is equal to the best constant $B_{p,p}(\Omega)$ (see, for example, \cite{BCT15}) in the $p$-Poincar\'e-Sobolev inequality
$$
\inf_{c \in \mathbb{R}} ||f-c\,|\,L_p(\Omega)|| \leq B_{p,p}(\Omega) ||\nabla f\,|\,L_p(\Omega)||,  \quad f \in W^{1}_{p}(\Omega).  
$$ 

\vskip 0.2cm

{\sc 1.4 Methods.}
The suggested method is based on the composition operators theory in relation with Brennan's Conjecture that allows to obtain universal weighted Poincar\'e-Sobolev inequalities 
in any simple connected domain $\Omega \neq \mathbb R^2$ for quasiconformal weights $h(x):=|J(x,\varphi)|$
generated by quasiconformal homeomorphisms $\varphi : \Omega \to \mathbb D$ (Theorem~\ref{T2.2}). In quasiconformal regular domains these weighted inequalities imply non-weighted Poincar\'e-Sobolev inequalities.
This method based on the theory of composition operators \cite{U93,VU02}
and its applications to the Sobolev type embedding theorems \cite{GG94,GU09}. 

The following diagram illustrates this idea:

\[\begin{array}{rcl}
W^{1}_{p}(\Omega) & \stackrel{\varphi^*}{\longrightarrow} & W^{1}_{q}(\mathbb D) \\[2mm]
\multicolumn{1}{c}{\downarrow} & & \multicolumn{1}{c}{\downarrow} \\[1mm]
L_s(\Omega) & \stackrel{(\varphi^{-1})^*}{\longleftarrow} & L_r(\mathbb D).
\end{array}\]

Here the operator $\varphi^*$ defined by the composition rule
$\varphi^*(f)=f \circ \varphi$ is a bounded composition operator on Sobolev
spaces induced by a homeomorphism $\varphi$ of $\Omega$ and $\mathbb D$ and
the operator $(\varphi^{-1})^*$ defined by the composition rule
$(\varphi^{-1})^*(f)=f \circ \varphi^{-1}$ is a bounded composition operator on
Lebesgue spaces. This method allows to transfer Poincar\'e-Sobolev inequalities 
from regular domains (for example, from the unit disc $\mathbb D$) to $\Omega$.

\vskip 0.3cm
{\bf Theorem B.}
\textit{Let $\Omega\subset\mathbb R^2$ be a $K$-quasiconformal $\alpha$-regular domain and $\varphi : \Omega \to \mathbb D$ be a $K$-quasiconformal mapping. Suppose that the Brennan's Conjecture holds.  
Then for every
$$
p \in \left(\max \left\{\frac{4K}{2K+1},\frac{\alpha (2K-1)+2}{\alpha K}\right\} \,, 2\right)
$$ 
the $p$-Poincar\'e-Sobolev inequality 
$$
\inf_{c \in \mathbb{R}} ||f-c\,|\,L_p(\Omega)|| \leq B_{p,p}(\Omega) ||\nabla f\,|\,L_p(\Omega)||,  \quad f \in W^{1}_{p}(\Omega),  
$$
holds with the constant
$$
B^p_{p,p}(\Omega) \leq 
\inf_{q \in I}
\bigg\{
B^p_{\frac{\alpha p}{\alpha -2},q}(\mathbb D) 
\cdot K \big\||D\varphi^{-1}|^{p-2}\mid L_{\frac{q}{p-q}}(\mathbb D)\big\| \cdot \big\|J_{\varphi^{-1}}\mid L_{\frac{\alpha}{2}}(\mathbb D)\big\|
\bigg\}.
$$ }

\begin{remark}
In the Introduction we formulate main results under the assumptions that the Brennan\'s conjecture holds true
$4K/(2K+1)<\beta <4K/(2K-1)$. In the main part of the paper we proved main results for
$4K/(2K+1)<\beta <2K\beta_0/(\beta_0(K-1)+2)$ for $\beta_0$ that is a recent value for which  the Brennan\'s conjecture is proved.
\end{remark}

\vskip 0.3cm

The next main theorem establish a connection between Brennan's Conjecture and composition operators on  Sobolev spaces:
\vskip 0.3cm
\textbf{Theorem C.}
\textit{Let $\Omega\subset\mathbb R^2$ be a simply connected domain. Generalized Brennan's Conjecture holds for a number $\beta \in (4K/(2K+1),\,4K/(2K-1))$
if and only if any $K$-quasiconformal  homeomorphism $\varphi : \Omega \to \mathbb D$ 
induces a bounded composition operator
$$
\varphi^*: L_p^1(\mathbb D) \to L_q^1(\Omega)
$$ 
for any $p \in (2\,, +\infty)$ and $q=p\beta/(p+\beta -2)$}.

\vskip 0.3cm

In the recent works we study composition operators on Sobolev spaces defined on
planar domains in connection with the conformal mappings theory \cite{GU12}. This connection
leads to weighted Sobolev  embeddings \cite{GU13,GU14} with the universal conformal weights.
Another application of conformal composition operators was given in \cite{BGU} where the 
spectral stability problem for conformal regular domains was considered.

\section{Composition operators and quasiconformal mappings}

In this section we recall basic facts about composition operators on Lebesgue and Sobolev spaces and the quasiconformal mappings theory. Let $\Omega\subset\mathbb R^n$, $n\geq 2$, be a $n$-dimensional Euclidean domain.
For any $1\leq p<\infty$ we consider the Lebesgue space $L_p(\Omega)$ of measurable functions $f: \Omega \to \mathbb{R}$ equipped with the following norm:
\[
\|f\mid L_p(\Omega)\|=\biggr(\int\limits _{\Omega}|f(x)|^{p}\, dx\biggr)^{\frac{1}{p}}<\infty.
\]  

The following theorem about composition operators on Lebesgue spaces is well known (see, for example \cite{VU02}):
\begin{theorem}
Let $\varphi :\Omega \to \widetilde{\Omega}$ be a weakly differentiable homeomorphism between two domains $\Omega$ and $\widetilde{\Omega}$.
Then the composition operator
\[
\varphi^{*}: L_r(\widetilde{\Omega}) \to L_s(\Omega),\,\,\,1 \leq s \leq r< \infty,
\] 
is bounded, if and only if $\varphi^{-1}$ possesses the Luzin $N$-property and
\[
\biggr(\int\limits _{\widetilde{\Omega}}\big(|J(y,\varphi^{-1})|\big)^{\frac{r}{r-s}}\, dy\biggr)^{\frac{r-s}{rs}}=K< \infty,\,\,\,1 \leq s<r< \infty,
\]
\[
\esssup\limits_{y \in \widetilde{\Omega}}\big(|J(y,\varphi^{-1})|\big)^{\frac{1}{s}}=K< \infty,\,\,\,1 \leq s=r< \infty.
\] 
The norm of the composition operator $\|\varphi^{*}\|=K$.
\end{theorem} 

We consider the Sobolev space $W^1_p(\Omega)$, $1\leq p<\infty$,
as a Banach space of locally integrable weakly differentiable functions
$f:\Omega\to\mathbb{R}$ equipped with the following norm: 
\[
\|f\mid W^1_p(\Omega)\|=\biggr(\int\limits _{\Omega}|f(x)|^{p}\, dx\biggr)^{\frac{1}{p}}+
\biggr(\int\limits _{\Omega}|\nabla f(x)|^{p}\, dx\biggr)^{\frac{1}{p}}.
\]
Recall that the Sobolev space $W^1_p(\Omega)$ coincides with the closer of the space of smooth functions $C^{\infty}(\Omega)$ in the norm of $W^1_p(\Omega)$.

We consider also the homogeneous seminormed Sobolev space $L^1_p(\Omega)$, $1\leq p<\infty$,
of locally integrable weakly differentiable functions $f:\Omega\to\mathbb{R}$ equipped
with the following seminorm: 
\[
\|f\mid L^1_p(\Omega)\|=\biggr(\int\limits _{\Omega}|\nabla f(x)|^{p}\, dx\biggr)^{\frac{1}{p}}.
\]

Recall that the embedding operator $i:L^1_p(\Omega)\to L_{1,\loc}(\Omega)$
is bounded.

By the standard definition functions of $L^1_p(\Omega)$ are defined only up to a set of measure zero, but they can be redefined quasieverywhere i.~e. up to a set of $p$-capacity zero. Indeed, every function $u\in L^1_p(\Omega)$ has a unique quasicontinuous representation $\tilde{u}\in L^1_p(\Omega)$. A function $\tilde{u}$ is termed quasicontinuous if for any $\varepsilon >0$ there is an open  set $U_{\varepsilon}$ such that the $p$-capacity of $U_{\varepsilon}$ is less then $\varepsilon$ and on the set $\Omega\setminus U_{\varepsilon}$ the function  $\tilde{u}$ is continuous (see, for example \cite{HKM,M}).

Let $\Omega$ and $\widetilde{\Omega}$ be domains in $\mathbb R^n$. We say that
a homeomorphism $\varphi:\Omega\to\widetilde{\Omega}$ induces a bounded composition
operator 
\[
\varphi^{\ast}:L^1_p(\widetilde{\Omega})\to L^1_q(\Omega),\,\,\,1\leq q\leq p\leq\infty,
\]
by the composition rule $\varphi^{\ast}(f)=f\circ\varphi$, if the composition $\varphi^{\ast}(f)\in L^1_q(\Omega)$
is defined quasi-everywhere in $\Omega$ and there exists a constant $K_{p,q}(\Omega)<\infty$ such that 
\[
\|\varphi^{\ast}(f)\mid L^1_q(\Omega)\|\leq K_{p,q}(\Omega)\|f\mid L^1_p(\widetilde{\Omega})\|
\]
for
any function $f\in L^1_p(\widetilde{\Omega})$ \cite{VU04}.

Let $\Omega\subset\mathbb R^n$ be an open set. A mapping $\varphi:\Omega\to\mathbb R^n$ belongs to $L^1_{p,\loc}(\Omega)$, 
$1\leq p\leq\infty$, if its coordinate functions $\varphi_j$ belong to $L^1_{p,\loc}(\Omega)$, $j=1,\dots,n$.
In this case the formal Jacobi matrix
$D\varphi(x)=\left(\frac{\partial \varphi_i}{\partial x_j}(x)\right)$, $i,j=1,\dots,n$,
and its determinant (Jacobian) $J(x,\varphi)=\det D\varphi(x)$ are well defined at
almost all points $x\in \Omega$. The norm $|D\varphi(x)|$ of the matrix
$D\varphi(x)$ is the norm of the corresponding linear operator $D\varphi (x):\mathbb R^n \rightarrow \mathbb R^n$ defined by the matrix $D\varphi(x)$.

Let $\varphi:\Omega\to\widetilde{\Omega}$ be weakly differentiable in $\Omega$. The mapping $\varphi$ is the mapping of finite distortion if $|D\varphi(z)|=0$ for almost all $x\in Z=\{z\in\Omega : J(x,\varphi)=0\}$.

A mapping $\varphi:\Omega\to\mathbb R^n$ possesses the Luzin $N$-property if a image of any set of measure zero has measure zero.
Mote that any Lipschitz mapping possesses the Luzin $N$-property.

The following theorem gives the analytic description of composition operators on Sobolev spaces:

\begin{theorem}
\label{CompTh} \cite{U93,VU02} A homeomorphism $\varphi:\Omega\to\widetilde{\Omega}$
between two domains $\Omega$ and $\widetilde{\Omega}$ induces a bounded composition
operator 
\[
\varphi^{\ast}:L^1_p(\widetilde{\Omega})\to L^1_q(\Omega),\,\,\,1\leq q< p<\infty,
\]
 if and only if $\varphi\in W_{1,\loc}^{1}(\Omega)$, has finite distortion,
and 
$$
K_{p,q}(\Omega)=\left(\int\limits_\Omega \left(\frac{|D\varphi(x)|^p}{|J(x,\varphi)|}\right)^\frac{q}{p-q}~dx\right)^\frac{p-q}{pq}<\infty.
$$
\end{theorem}

Recall that a homeomorphism $\varphi: \Omega\to \widetilde{\Omega}$ is called a $K$-quasiconformal mapping if $\varphi\in W^1_{n,\loc}(\Omega)$ and there exists a constant $1\leq K<\infty$ such that
$$
|D\varphi(x)|^n\leq K |J(x,\varphi)|\,\,\text{for almost all}\,\,x\in\Omega.
$$

Quasiconformal mappings have a finite distortion, i.~e.  $D\varphi(x)=0$ for almost all points $x$
that belongs to set $Z=\{x\in \Omega:J(x,\varphi)=0\}$ and any quasiconformal mapping possesses Luzin $N$-property.  A mapping which is inverse to a quasiconformal mapping is also quasiconformal.

If $\varphi : \Omega \to \widetilde{\Omega}$ is a $K$-quasiconformal mapping then $\varphi$ is differentiable almost everywhere in $\Omega$ and
$$
|J(x,\varphi)|=J_{\varphi}(x):=\lim\limits_{r\to 0}\frac{|\varphi(B(x,r))|}{|B(x,r)|}\,\,\text{for almost all}\,\,x\in\Omega.
$$   

Note, that a homeomorphism $\varphi: \Omega\to \widetilde{\Omega}$ is a $K$-quasiconformal mapping if and only if $\varphi$ generates by the composition rule $\varphi^{\ast}(f)=f\circ\varphi$ an isomorphism of Sobolev spaces $L^1_n(\Omega)$ and $L^1_n(\widetilde{\Omega})$:
$$
K^{-\frac{1}{n}}\|f \mid L^1_n(\widetilde{\Omega})\|\leq \|\varphi^{\ast}f \mid L^1_n(\Omega)\|\leq K^{\frac{1}{n}}\|f \mid L^1_n(\widetilde{\Omega})\|
$$
for any $f\in L^1_n(\widetilde{\Omega})$ \cite{VG75}.

Compositions of Sobolev functions of the spaces $L^1_p(\Omega')$, $p\ne n$, with quasiconformal mappings was studied also in \cite{KR}.

For any planar $K$-quasiconformal homeomorphism $\varphi : \Omega \to \widetilde{\Omega}$,
the following sharp results is known: $J(x,\varphi) \in L_{\alpha^*, \loc}(\widetilde{\Omega})$ 
for any $\alpha^*<K/(K-1)$ (\cite{Ast}).

If $K\equiv 1$ then $1$-quasiconformal homeomorphisms are conformal mappings and in the space $\mathbb R^n$, $n\geq 3$, are exhausted by M\"obius transformations.  

\section{Composition operators and Brennan's Conjecture}

Brennan's Conjecture \cite{Br} is that if $\varphi:\Omega\to\mathbb D$ is a conformal mappings of a simply connected planar domain $\Omega$, $\Omega\ne \mathbb R^2$, onto the unit disc $\mathbb D$ then
\begin{equation}
\int\limits _{\Omega}|\varphi'(x)|^{\beta}~dx<+\infty,\quad\text{for all}\quad\frac{4}{3}<\beta<4.\label{eq:BR1.1}
\end{equation}

For $4/3<s<3$, it is a comparatively easy consequence of the Koebe distortion theorem (see, for example, \cite{Ber}).
J.~Brennan \cite{Br} (1973) extended this range to $4/3<s<3+\delta$, where $\delta>0$, and conjectured it to hold 
for $4/3<s<4$. The example of
$\Omega=\mathbb{C}\setminus(-\infty,-1/4]$ shows that this range
of $s$ cannot be extended.

Brennan's Conjecture proved for $\beta\in (4/3,\beta_0)$, were $\beta_0=3.752$ \cite{HShim}. Brennan's Conjecture for quasiconformal mappings was considered in \cite{HSS}. In \cite{HSS} was proved, that if $\varphi:\Omega\to\mathbb D$ be a $K$-quasiconformal mapping, then
\begin{equation}
\int\limits _{\Omega}|D\varphi(x)|^{\beta}~dx<+\infty,\quad\text{for all}\quad\frac{4K}{2K+1}<\beta<\frac{2K\beta_0}{(K-1)\beta_0+2}.\label{eq:HSS1.1}
\end{equation}
Here $\beta_0$ is the proved upper bound for Brennan's Conjecture.

Now we prove, that Generalized Brennan's Conjecture leads to boundedness of composition operators on Sobolev spaces generates by quasiconformal mappings.

\vskip 0.3cm

\textbf{Theorem C.}
\textit{Let $\Omega\subset\mathbb R^2$ be a simply connected domain. Generalized Brennan's Conjecture holds for a number $\beta \in (4K/(2K+1),\,2K\beta_0/(\beta_0(K-1)+2))$
if and only if any $K$-quasiconformal  homeomorphism $\varphi : \Omega \to \mathbb D$ 
induces a bounded composition operator
$$
\varphi^*: L_p^1(\mathbb D) \to L_q^1(\Omega)
$$ 
for any $p \in (2\,, +\infty)$ and $q=p\beta/(p+\beta -2)$}.

\begin{proof}
By the composition theorem \cite{U93,VU02} a homeomorphism $\varphi: \Omega\to\mathbb D$
induces a bounded composition operator
$$
\varphi^*: L_p^1(\mathbb D) \to L_q^1(\Omega), \quad 1\leq q<p<\infty.
$$ 
if and only if $\varphi\in W^1_{1,\loc}(\Omega)$, has finite distortion and 
$$
K_{p,q}(\Omega)=\left(\int\limits_\Omega \left(\frac{|D\varphi(x)|^p}{|J(x,\varphi)|}\right)^\frac{q}{p-q}~dx\right)^\frac{p-q}{pq}<\infty.
$$
Because $\varphi$ is a quasiconformal mapping, then $\varphi\in W^1_{n,\loc}(\Omega)$ and Jacobian $J(x,\varphi)\ne 0$ for almost all $x\in\Omega$. Hence the $p$-dilatation
$$
K_p(x)=\frac{|D\varphi(x)|^p}{|J(x,\varphi)|}
$$
is well defined for almost all $x\in\Omega$ and so $\varphi$ is a mapping of finite distortion.

By Brennan's Conjecture
$$
\int\limits _{\Omega}|D\varphi(x)|^{\beta}~dx <+\infty,\quad\text{for all}\quad\frac{4K}{2K+1}<\beta<\frac{2K\beta_0}{(K-1)\beta_0+2}.
$$
Then
\begin{multline*}
K_{p,q}^{\frac{pq}{p-q}}(\Omega)=\int\limits_\Omega \left(\frac{|D\varphi(x)|^p}{|J(x,\varphi)|}\right)^\frac{q}{p-q}~dx=
\int\limits_\Omega \left(\frac{|D\varphi(x)|^2}{|J(x,\varphi)|}|D\varphi(x)|^{p-2}\right)^\frac{q}{p-q}~dx\\
\leq 
K^{\frac{q}{p-q}}\int\limits_\Omega \left(|D\varphi(x)|^{p-2}\right)^\frac{q}{p-q}~dx =
K^{\frac{q}{p-q}}\int\limits_\Omega |D\varphi(x)|^{\beta}~dx<\infty,
\end{multline*}
for $\beta=(p-2)q/(p-q)$. Hence we have a bounded composition operator
$$
\varphi^*: L_p^1(\mathbb D) \to L_q^1(\Omega)
$$ 
for any $p \in (2\,, +\infty)$ and $q=p\beta/(p+\beta -2)$.

Let us check that $q<p$. Because $p>2$ we have that $p+\beta -2> \beta >0$ and so $\beta /(p+\beta -2) <1$. Hence we obtain  $q<p$. 

Now, let the composition operator
$$
\varphi^*: L_p^1(\mathbb D) \to L_q^1(\Omega),\,\,q<p,
$$
is bounded for any $p \in (2\,, +\infty)$ and $q=p\beta/(p+\beta -2)$.
Then, using the  Hadamard inequality: 
$$
|J(x,\varphi)|\leq |D\varphi(x)|^2\,\,\text{for almost all}\,\, x\in\Omega,
$$ 
and Theorem~\ref{CompTh}, we have
$$
\int\limits _{\Omega}|D\varphi(x)|^{\beta}~dx=\int\limits _{\Omega}|D\varphi(x)|^{\frac{(p-2)q}{p-q}}~dx\leq 
\int\limits_\Omega \left(\frac{|D\varphi(x)|^{p}}{|J(x,\varphi)|}\right)^\frac{q}{p-q}~dx<+\infty. 
$$

\end{proof}

The suggested approach to the Poincar\'e-Sobolev type inequalities in bounded planar domains $\Omega\subset\mathbb R^2$ is based on translation of these inequalities from the unit disc $\mathbb D$ to $\Omega$. On this way we use the following duality \cite{VU98}: 

\begin{theorem}
\label{theorem:duality}
Let a homeomorphism $\varphi: \Omega\to\Omega'$, $\Omega,\Omega'\subset\mathbb R^2$, generates by the composition rule 
$\varphi^{\ast}(f)=f\circ\varphi$ a bounded composition operator  
$$
\varphi^{\ast}: L^1_p(\Omega')\to L^1_q(\Omega), \,\,\,1< q\leq p<\infty,
$$
then the inverse mapping $\varphi^{-1}:\Omega'\to\Omega$ generates by the composition rule 
$(\varphi^{-1})^{\ast}(g)=g\circ\varphi^{-1}$ a bounded composition operator
$$
(\varphi^{-1})^{\ast}: L^1_{q'}(\Omega)\to L^1_{p'}(\Omega'), \,\,\,\frac{1}{q}+\frac{1}{q'}=1, \frac{1}{p}+\frac{1}{p'}=1.
$$
\end{theorem}

From Theorem~C and Theorem~\ref{theorem:duality} immediately follow Theorem~\ref{T2.1}:

\begin{theorem}\label{T2.1}
Let $\Omega\subset\mathbb R^2$ be a simply connected domain and $\varphi : \Omega \to \mathbb D$ be a $K$-quasiconformal  homeomorphism. 
Suppose that $p \in (2K\beta_0/((K+1)\beta_0-2)\,, 2)$.

Then the inverse mapping $\varphi^{-1}$ induces a bounded composition operator
$$
(\varphi^{-1})^*: L_p^1(\Omega) \to L_q^1(\mathbb D)
$$ 
for any $q$ such that
$$
1\leq q \leq \frac{(2\beta_0-4)p}{2K\beta_0-((K-1)\beta_0+2)p} < \frac{2p}{4K-(2K-1)p}.
$$
The inequality
\begin{equation}\label{Eq_1}
||(\varphi^{-1})^*f\,|\,L^{1}_{q}(\mathbb D)|| \leq 
K^{\frac{1}{p}}\left(\int\limits_\mathbb D |D\varphi^{-1}(y)|^{\frac{(p-2)q}{p-q}}dy \right)^{\frac{p-q}{pq}}
||f\,|\,L^{1}_{p}(\Omega)||
\end{equation}
holds for any function $f \in L^{1}_{p}(\Omega)$.
\end{theorem}

\begin{proof}
By Theorem~C we have a bounded composition operator
$$
\varphi^{\ast}: L^1_{q'}(\mathbb D)\to L^1_{p'}(\Omega)
$$
for $q' \in (2, +\infty)$ and $p'=q'\beta/(q'+\beta -2)$. 

Then, by Brennan's Conjecture, 
$p' \in (2\,, 2K\beta_0/((K-1)\beta_0+2))$. Now using Theorem~\ref{theorem:duality} we have
$$
p=\frac{p'}{p'-1} \in \left(\frac{2K\beta_0}{(K+1)\beta_0-2}\,,2\right).
$$

Since
$$
p=\frac{p'}{p'-1}=\frac{q' \beta}{q' \beta-(q'+\beta+2)},
$$
we obtain by direct calculations that
$$
q'=\frac{(4-2 \beta_0)p}{2K \beta_0-((K+1)\beta_0-2)p}.
$$

By Theorem~\ref{theorem:duality} $q=q'/(q'-1)$ and $q \leq p$. By elementary calculations
$$
1\leq q \leq \frac{(2\beta_0-4)p}{2K\beta_0-((K-1)\beta_0+2)p} < \frac{2p}{4K-(2K-1)p}. 
$$

Now we prove the inequality~(\ref{Eq_1}).
Let $f \in L^{1}_{p}(\Omega)\cap C^{\infty}(\Omega)$. Then the composition $g=(\varphi^{-1})^*(f) \in L_{1,\loc}^1(\mathbb D)$ \cite{VGR}.
Hence, using Theorem~\ref{CompTh} we obtain
\begin{multline*}
||g\,|\,L^{1}_{q}(\mathbb D)|| \leq 
\left(\int\limits_\mathbb D\left(\frac{|D\varphi^{-1}(y)|^p}{|J(y,\varphi^{-1})|}\right)^{\frac{q}{p-q}}dy \right)^{\frac{p-q}{pq}}
||f\,|\,L^{1}_{p}(\Omega)|| \\
{} = \left(\int\limits_\mathbb D\left(\frac{|D\varphi^{-1}(y)|^2 \cdot |D\varphi^{-1}(y)|^{p-2}}{|J(y,\varphi^{-1})|}\right)^{\frac{q}{p-q}}dy \right)^{\frac{p-q}{pq}}
||f\,|\,L^{1}_{p}(\Omega)|| \\
{} \leq K^{\frac{1}{p}}\left(\int\limits_\mathbb D |D\varphi^{-1}(y)|^{\frac{(p-2)q}{p-q}}dy \right)^{\frac{p-q}{pq}}
||f\,|\,L^{1}_{p}(\Omega)||.
\end{multline*} 

Approximating an arbitrary function $f \in L_p^1(\Omega)$ by smooth functions \cite{VU98,VU02}, 
we obtain the required inequality.   
\end{proof}

\section{Poincar\'e-Sobolev inequalities}

\textbf{Weighted Poincar\'e-Sobolev inequalities}. Let $\Omega \subset \mathbb R^2$
be a planar domain and let $v : \Omega \to \mathbb R$ be a real valued
function, $v>0$ a.~e. in $\Omega$. We consider the weighted Lebesgue space $L_p(\Omega,v)$, $1\leq p<\infty$, 
of measurable functions $f: \Omega \to \mathbb R$  with the finite norm
$$
\|f\,|\,L_{p}(\Omega,v)\|:= \left(\int\limits_\Omega|f(x)|^pv(x)dx \right)^{\frac{1}{p}}< \infty.
$$
It is a Banach space for the norm $\|f\,|\,L_{p}(\Omega,v)\|$.

Using Theorem~\ref{T2.1} we prove

\begin{theorem}\label{T2.2}
Suppose that $\Omega\subset\mathbb R^2$ is a simply connected domain and $h(x) =|J(x,\varphi)|$ is 
the quasiconformal weight defined by a $K$-quasiconformal homeomorphism 
$\varphi : \Omega \to \mathbb D$.Then for every $p \in (2K\beta_0/((K+1)\beta_0-2)\,, 2)$ and every function $f \in W^{1}_{p}(\Omega)$,
the inequality
$$
\inf\limits_{c \in \mathbb R}\left(\int\limits_\Omega |f(x)-c|^rh(x)dx\right)^{\frac{1}{r}} \leq B_{r,p}(\Omega,h)
\left(\int\limits_\Omega |\nabla f(x)|^p dx\right)^{\frac{1}{p}}
$$ 
holds for any $r$ such that
$$
1 \leq r < \frac{p}{2-p} \cdot \frac{2\beta_0-4}{K\beta_0}
$$ 
with the constant
$$
B_{r,p}(\Omega,h) \leq \inf_{q \in [1,2p/(4K-(2K-1)p))}\{\widetilde{K}_{p,q}(\mathbb D) \cdot B_{r,q}(\mathbb D) \}.
$$
\end{theorem}

Here $B_{r,q}(\mathbb D)$ is the best constant in the (non-weight) Poincar\'e-Sobolev inequality in the unit disc
$\mathbb D \subset \mathbb R^2$ and 
$$
\widetilde{K}_{p,q}(\mathbb D) = K^{\frac{1}{p}} \left(\int\limits_\mathbb D |D\varphi^{-1}(y)|^{\frac{(p-2)q}{p-q}}dy \right)^{\frac{p-q}{pq}}. 
$$ 

\begin{proof} 
By conditions of the theorem there exists a $K$-quasiconformal mapping $\varphi : \Omega \to \mathbb D$ such that
$$
\int\limits _{\Omega}|D\varphi(x)|^{\beta}~dx<+\infty,\quad\text{for all}\quad\frac{4K}{2K+1}<\beta<\frac{2K\beta_0}{(K-1)\beta_0+2}.
$$ 
By Theorem \ref{T2.1}, if
\begin{equation}\label{IN2.2}
1\leq q \leq \frac{(2\beta_0-4)p}{2K\beta_0-((K-1)\beta_0+2)p} < \frac{2p}{4K-(2K-1)p} 
\end{equation}
then the inequality  
\begin{equation}\label{IN2.1}
||\nabla (f \circ \varphi^{-1}) \,|\, L_{q}(\mathbb D)|| \leq \widetilde{K}_{p,q}(\mathbb D) ||\nabla f \,|\, L_{p}(\Omega)||
\end{equation}
holds for every function $f \in L^{1}_{p}(\Omega)$.

Let $f \in L^{1}_{p}(\Omega) \cap C^1(\Omega)$. Then the function $g=f \circ \varphi^{-1}$ is defined almost everywhere in $\mathbb D$ and belongs to the Sobolev space $L^{1}_{q}(\mathbb D)$ \cite{VGR}. Hence, by the Sobolev embedding theorem $g=f \circ \varphi^{-1} \in W^{1,q}(\mathbb D)$ \cite{M}
and the classical Poincar\'e-Sobolev inequality,
\begin{equation}\label{IN2.3}
\inf_{c \in \mathbb R}||f \circ \varphi^{-1} -c \,|\, L_{r}(\mathbb D)|| \leq B_{r,q}(\mathbb D) ||\nabla (f \circ \varphi^{-1}) \,|\, L_{q}(\mathbb D)||
\end{equation}
holds for any $r$ such that
\begin{equation}
\label{ineqr}
1 \leq r \leq \frac{2q}{2-q}.
\end{equation}
By elementary calculations from the inequality \eqref{IN2.2}, it follows that
\begin{equation}
\label{ineqq}
\frac{2q}{2-q} \leq \frac{\beta_0-2}{K\beta_0} \cdot \frac{2p}{2-p} < \frac{1}{K} \cdot \frac{p}{2-p}. 
\end{equation} 

Combining inequalities (\ref{ineqr}) and (\ref{ineqq}) we conclude that the inequality \eqref{IN2.3}
holds for any $r$ such that
$$
1 \leq r \leq \frac{\beta_0-2}{K\beta_0} \cdot \frac{2p}{2-p} < \frac{1}{K} \cdot \frac{p}{2-p}.
$$ 

Using the change of variable formula for quasiconformal mappings \cite{VGR}, the classical Poincar\'e-Sobolev inequality for the unit disc 
$$
\inf\limits_{c \in \mathbb R}\left(\int\limits_{\mathbb D} |g(y)-c|^rdy\right)^{\frac{1}{r}} \leq B_{r,q}(\mathbb D)
\left(\int\limits_{\mathbb D} |\nabla g(y)|^q dy\right)^{\frac{1}{q}}
$$
and inequality \eqref{IN2.1}, we finally infer 
\begin{multline*}
\inf\limits_{c \in \mathbb R}\left(\int\limits_\Omega |f(x)-c|^rh(x)dx\right)^{\frac{1}{r}} =
\inf\limits_{c \in \mathbb R}\left(\int\limits_\Omega |f(x)-c|^r |J(x,\varphi)| dx\right)^{\frac{1}{r}} \\
{} = \inf\limits_{c \in \mathbb R}\left(\int\limits_{\mathbb D} |g(y)-c|^rdy\right)^{\frac{1}{r}} \leq B_{r,q}(\mathbb D) 
\left(\int\limits_{\mathbb D} |\nabla g(y)|^q dy\right)^{\frac{1}{q}} \\
{} \leq \widetilde{K}_{p,q}(\mathbb D) B_{r,q}(\mathbb D) 
\left(\int\limits_{\Omega} |\nabla f(x)|^p dx\right)^{\frac{1}{p}}.
\end{multline*}

Approximating an arbitrary function $f \in W^{1}_{p}(\Omega)$ by smooth functions we have
$$
\inf\limits_{c \in \mathbb R}\left(\int\limits_\Omega |f(x)-c|^rh(x)dx\right)^{\frac{1}{r}} \leq
B_{r,p}(\Omega,h) \left(\int\limits_{\Omega} |\nabla f(x)|^p dx\right)^{\frac{1}{p}}
$$
with the constant
$$
B_{r,p}(\Omega,h) \leq \inf_{q \in [1,2p/(4K-(2K-1)p))}\{\widetilde{K}_{p,q}(\mathbb D) \cdot B_{r,q}(\mathbb D)\}.
$$
\end{proof}

The property of the quasiconformal $\alpha$-regularity implies the integrability of a 
Jacobian of quasiconformal mappings and therefore for any quasiconformal $\alpha$-regular domain
we have the embedding of weighted Lebesgue spaces $L_r(\Omega,h)$ into non-weighted Lebesgue
spaces $L_s(\Omega)$ for $s=\frac{\alpha -2}{\alpha}r$.

\begin{lemma} \label{L2.3}
Let $\Omega$ be a $K$-quasiconformal $\alpha$-regular domain.Then for any function
$f \in L_r(\Omega,h)$, $\alpha / (\alpha - 2) \leq r < \infty$, the inequality
$$
||f\,|\,L_s(\Omega)|| \leq \left(\int\limits_\mathbb D \big|J(y,\varphi^{-1})\big|^{\frac{\alpha}{2}}~dy \right)^{{\frac{2}{\alpha}} \cdot \frac{1}{s}} ||f\,|\,L_r(\Omega,h)||
$$
holds for $s=\frac{\alpha -2}{\alpha}r$.
\end{lemma}

\begin{proof}
By the assumptions of the lemma these exists a $K$-quasiconformal mapping $\varphi : \Omega \to \mathbb D$ such that
$$
\int\limits_\mathbb D \big|J(y,\varphi^{-1})\big|^{\frac{\alpha}{2}}~dy  < +\infty,
$$
Let $s=\frac{\alpha -2}{\alpha}r$. 
Then using the change of variable formula for quasiconformal mappings \cite{VGR}, H\"older's inequality with exponents $(r,rs/(r-s))$ and equality 
$|J(x,\varphi)| = h(x)$, we obtain
\begin{multline*}
||f\,|\,L_s(\Omega)|| \\
{} = \left(\int\limits_{\Omega} |f(x)|^s dx\right)^{\frac{1}{s}} = 
\left(\int\limits_{\Omega} |f(x)|^s \big|J(x,\varphi)\big|^{\frac{s}{r}} \big|J(x,\varphi)\big|^{-\frac{s}{r}}~dx\right)^{\frac{1}{s}} \\
{} \leq \left(\int\limits_{\Omega} |f(x)|^r |J(x,\varphi)| dx\right)^{\frac{1}{r}}
 \left(\int\limits_{\Omega} \big|J(x,\varphi)\big|^{- \frac{s}{r-s}}~ dx\right)^{\frac{r-s}{rs}} \\
{} \leq \left(\int\limits_{\Omega} |f(x)|^r h(x)~ dx\right)^{\frac{1}{r}} 
\left(\int\limits_{\mathbb D} \big|J(y,\varphi^{-1})\big|^{\frac{r}{r-s}}~ dy\right)^{\frac{r-s}{rs}} \\
{} = \left(\int\limits_{\Omega} |f(x)|^r h(x)~ dx\right)^{\frac{1}{r}}
\left(\int\limits_\mathbb D \big|J(y,\varphi^{-1})\big|^{\frac{\alpha}{2}}~dy \right)^{{\frac{2}{\alpha}} \cdot \frac{1}{s}}.
\end{multline*}
 
\end{proof}

The following theorem gives the upper estimate of the Poincar\'e constant and follows from Theorem \ref{T2.2} and Lemma \ref{L2.3}:

\begin{theorem}\label{T2.4}
Suppose that $\Omega \subset \mathbb C$ is a $K$-quasiconformal $\alpha$-regular domain.
Then for every
$$
p \in \left(\max \left\{\frac{4K}{2K+1},\frac{2K \alpha \beta_0}{(K+2) \alpha \beta_0-4(\alpha + \beta_0 -2)}\right\} \,, 2\right),
$$ 
every
$$
s\in \left[1,\, \frac{(\alpha-2)(\beta_0 -2)}{K \alpha \beta_0} \cdot \frac{2p}{2-p}\right]
$$
and every function $f \in W^{1}_{p}(\Omega)$, the inequality 
$$
\inf\limits_{c \in \mathbb R}\left(\int\limits_\Omega |f(x)-c|^sdx\right)^{\frac{1}{s}} \leq B_{s,p}(\Omega)
\left(\int\limits_\Omega |\nabla f(x)|^p dx\right)^{\frac{1}{p}}
$$ 
holds with the constant
\begin{multline*}
B_{s,p}(\Omega) \leq \left(\int\limits_\mathbb D \big|J(y,\varphi^{-1})\big|^{\frac{\alpha}{2}}~dy \right)^{{\frac{2}{\alpha}} \cdot \frac{1}{s}} B_{r,p}(\Omega,h) \\
{} \leq \inf_{q \in I}\left\{B_{\frac{\alpha s}{\alpha -2},q}(\mathbb D)  \cdot \widetilde{K}_{p,q}(\mathbb D)
\cdot \left(\int\limits_\mathbb D \big|J(y,\varphi^{-1})\big|^{\frac{\alpha}{2}}~dy \right)^{{\frac{2}{\alpha}} \cdot \frac{1}{s}}\right\},
\end{multline*}
where  $I=[1,2p/(4K-(2K-1)p))$.
\end{theorem}

\begin{proof}
Let $f\in W^1_p(\Omega)$. Then by Theorem \ref{T2.2} and Lemma \ref{L2.3} we obtain
\begin{multline*}
\inf_{c \in \mathbb R} \left(\int\limits_{\Omega} |f(x)-c|^s dx\right)^{\frac{1}{s}} \\
{} \leq \left(\int\limits_\mathbb D \big|J(y,\varphi^{-1})\big|^{\frac{\alpha}{2}}~dy \right)^{{\frac{2}{\alpha}} \cdot \frac{1}{s}}
\inf_{c \in \mathbb R} \left(\int\limits_{\Omega} |f(x)-c|^r h(x) dx\right)^{\frac{1}{r}} \\
{} \leq B_{r,p}(\Omega,h) 
\left(\int\limits_\mathbb D \big|J(y,\varphi^{-1})\big|^{\frac{\alpha}{2}}~dy \right)^{{\frac{2}{\alpha}} \cdot \frac{1}{s}}
\left(\int\limits_\Omega |\nabla f(x)|^p dx\right)^{\frac{1}{p}}
\end{multline*} 
for $s\geq 1$.

Because by Lemma \ref{L2.3} $s=\frac{\alpha -2}{\alpha}r$ and by Theorem \ref{T2.2}
$$
1 \leq r \leq \frac{\beta_0-2}{K\beta_0} \cdot \frac{2p}{2-p} < \frac{1}{K} \cdot \frac{p}{2-p},
$$  
then
$$
1 \leq s \leq \frac{(\alpha-2)(\beta_0 -2)}{K \alpha \beta_0} \cdot \frac{2p}{2-p} < \frac{\alpha -2}{K\alpha} \cdot \frac{p}{2-p}. 
$$
Hence, by direct calculations, we obtain that
$$
p \geq \frac{2K \alpha \beta_0}{(K+2) \alpha \beta_0-4(\alpha + \beta_0 -2)} > \frac{2K \alpha}{(K+1)\alpha - 2}.
$$

The last inequality holds provided that the Brennan's conjecture holds true all
$\beta: 4K/(2K+1)< \beta < 4K/(2K-1)$.
\end{proof}

In the case of $(p,p)$-Poincar\'e-Sobolev inequalities we have:

\begin{theorem}\label{T3.4}
Let $\Omega\subset\mathbb R^2$ be a $K$-quasiconformal $\alpha$-regular domain and $\varphi : \Omega \to \mathbb D$ be a $K$-quasiconformal mapping. 
Then for every
$$
p \in \left(\max \left\{\frac{4K}{2K+1},\frac{2(K-1)\alpha \beta_0+4(\alpha+ \beta_0-2)}{K \alpha \beta_0}\right\} \,, 2\right)
$$  
the $p$-Poincar\'e-Sobolev inequality 
$$
\inf_{c \in \mathbb{R}} ||f-c\,|\,L_p(\Omega)|| \leq B_{p,p}(\Omega) ||\nabla f\,|\,L_p(\Omega)||,  \quad f \in W^{1}_{p}(\Omega),  
$$
holds with the constant
$$
B^p_{p,p}(\Omega)
 \leq 
\inf_{q \in I}
\left\{
B_{\frac{\alpha p}{\alpha -2},q}^p(\mathbb D)
\cdot \widetilde{K}_{p,q}^p(\mathbb D) 
\cdot \left(\int\limits_\mathbb D \big|J(y,\varphi^{-1})\big|^{\frac{\alpha}{2}}~dy \right)^{\frac{2}{\alpha}} 
\right\},
$$
where $I=[1,2p/(4K-(2K-1)p))$.
\end{theorem}

\begin{proof}
By Lemma \ref{L2.3} $p=\frac{\alpha -2}{\alpha}r$ and
by Theorem \ref{T2.2}
$$
1 \leq r \leq \frac{\beta_0-2}{K\beta_0} \cdot \frac{2p}{2-p} < \frac{1}{K} \cdot \frac{p}{2-p}.
$$ 
Hence
$$
 \frac{\alpha}{\alpha -2} \leq \frac{\beta_0-2}{K\beta_0} \cdot \frac{2}{2-p} < \frac{1}{K} \cdot \frac{1}{2-p}. 
$$
By elementary calculations
$$
p \geq \frac{2(K-1)\alpha \beta_0 +4(\alpha +\beta_0 -2)}{K \alpha \beta_0}>\frac{(2K-1)\alpha +2}{K \alpha}.
$$

The last inequality is correct by factor that Brennan's conjecture is correct for all
$\beta: 4K/(2K+1)< \beta < 4K/(2K-1)$.
\end{proof}

Theorem~\ref{T3.4} implies the lower estimates of the first non-trivial eigenvalue $\mu_p^{(1)}(\Omega)$:

\begin{theorem}\label{T3.5}
Let $\varphi : \Omega \to \mathbb D$ be a $K$-quasiconformal homeomorphism from a 
$K$-quasiconformal $\alpha$-regular domain $\Omega$ to the unit disc $\mathbb D$.  
Then for every
$$
p \in \left(\max \left\{\frac{4K}{2K+1},\frac{2(K-1)\alpha \beta_0+4(\alpha+ \beta_0-2)}{K \alpha \beta_0}\right\} \,, 2\right)
$$ 
the following inequality holds
$$
\frac{1}{\mu_p^{(1)}(\Omega)} 
 \leq 
 \inf_{q \in I}
\left\{
B_{\frac{\alpha p}{\alpha -2},q}^p(\mathbb D)
\cdot \widetilde{K}_{p,q}^p(\mathbb D) 
\cdot \left(\int\limits_\mathbb D \big|J(y,\varphi^{-1})\big|^{\frac{\alpha}{2}}~dy \right)^{\frac{2}{\alpha}} 
\right\},
$$
where $I=[1,2p/(4K-(2K-1)p))$.
\end{theorem}

\section{Estimates in the some domains}

$\mathbf{Example \, A.}$ The homeomorphism 
$$
w= Az+B \overline{z}, \quad z=x+iy, \quad A>B \geq 0,
$$
is $K$-quasiconformal with $K=\frac{A+B}{A-B}$ and maps the unit disc $\mathbb D$ 
onto the interior of the ellipse
$$
\Omega_e= \left\{(x,y) \in \mathbb R^2: \frac{x^2}{(A+B)^2}+\frac{y^2}{(A-B)^2}=1\right\}.
$$

We calculate the norm of the derivative of mapping $w$ by the formula 
$$
|Dw|=|w_z|+|w_{\overline{z}}|
$$
and the Jacobian of mapping $w$ by the formula 
$$
J(z,w)=|w_z|^2-|w_{\overline{z}}|^2.
$$
Here
$$
w_z=\frac{1}{2}\left(\frac{\partial w}{\partial x}-i\frac{\partial w}{\partial y}\right) \quad \text{and} \quad 
w_{\overline{z}}=\frac{1}{2}\left(\frac{\partial w}{\partial x}+i\frac{\partial w}{\partial y}\right).
$$ 

By elementary calculations
$$
w_z=A \quad \text{and} \quad w_{\overline{z}}=B.
$$ 
Hence
$$
|Dw|=A+B \quad \text{and} \quad J(z,w)=A^2-B^2.
$$

Then by Theorem~\ref{T3.5} we have
\begin{multline*}
\frac{1}{\mu_p^{(1)}(\Omega_e)} \leq \inf_{q \in I}
\left(\frac{2}{\pi^d}\left(\frac{1-d}{1/2-d}\right)^{1-d}\right)^p \\
{} \times
\frac{A+B}{A-B} \left(\int\limits_\mathbb D (A^2-B^2)^{\frac{\alpha}{2}}~dy \right)^{\frac{2}{\alpha}} 
\left(\int\limits_\mathbb D (A+B)^{\frac{(p-2)q}{p-q}}~dy \right)^{\frac{p-q}{q}} \\
{} = \inf_{q \in I}
\left(\frac{2}{\pi^d}\left(\frac{1-d}{1/2-d}\right)^{1-d}\right)^p
(A+B)^p \pi^{\frac{2q+ \alpha (p-q)}{\alpha q}}, 
\end{multline*}
where $I=[1,2p/(4K-(2K-1)p))$ and $d=1/q-(\alpha -2) \alpha p$.

\vskip 0.3cm
$\mathbf{Example \, B.}$ The homeomorphism 
$$
w= \left(|z|^{k-1}z+1\right)^2, \quad z=x+iy, \quad k\geq 1,
$$
is $k$-quasiconformal and maps the unit disc $\mathbb D$ onto the interior of the cardioid
$$
\Omega_c= \left\{(x,y) \in \mathbb R^2: (x^2+y^2-2x)^2-4(x^2+y^2)=0\right\}.
$$ 

We calculate the partial derivatives of mapping $w$
$$
w_z=(k+1)|z|^{k-1}\left(|z|^{k-1}z+1\right) \quad \text{and} \quad w_{\overline{z}}=(k-1)|z|^{k-3}z^2\left(|z|^{k-1}z+1\right).
$$ 
Hence
$$
|Dw|=2k|z|^{k-1}\sqrt{|z|^{2k}+|z|^{k-1}(z+\overline{z})+1}
$$
and
$$
J(z,w)=4k|z|^{2k-2}\left(|z|^{2k}+|z|^{k-1}(z+\overline{z})+1\right).
$$
Then by Theorem~\ref{T3.5} we have
\begin{multline*}
\frac{1}{\mu_p^{(1)}(\Omega_c)} \leq \inf_{q \in I} (2k)^p
\left(\frac{2}{\pi^d}\left(\frac{1-d}{1/2-d}\right)^{1-d}\right)^p \\
{} \times
\left(\int\limits_0^{2 \pi} \left(\int\limits_0^1 \left(\rho^{2k-2}\left(\rho^{2k}+2 \rho^k \cos \psi +1\right)\right)^{\frac{\alpha}{2}} \rho~d \rho\right)~d \psi \right)^{\frac{2}{\alpha}} \\
{} \times
\left(\int\limits_0^{2 \pi} \left(\int\limits_0^1 \left(\rho^{k-1} \sqrt{\rho^{2k}+2 \rho^k \cos \psi +1}\right)^{\frac{(p-2)q}{p-q}} \rho~d \rho\right)~d \psi \right)^{\frac{p-q}{q}}.
\end{multline*}
Here $I=[1,2p/(4k-(2k-1)p))$ and $d=1/q-(\alpha -2) \alpha p$.

\vskip 0.3cm
$\mathbf{Example \, C.}$ The homeomorphism 
$$
w= |z|^{k}z, \quad z=x+iy, \quad k\geq 0,
$$
is $(k+1)$-quasiconformal and maps the square 
$$
Q:=\left\{(x,y) \in \mathbb R^2:-\frac{\sqrt{2}}{2} <x< \frac{\sqrt{2}}{2},\, -\frac{\sqrt{2}}{2} <y< \frac{\sqrt{2}}{2}\right\}
$$ 
onto star-shaped domains $\Omega_{\varepsilon}^*$ with vertices $(\pm \sqrt{2}/2,\, \pm \sqrt{2}/2),
(\pm \varepsilon,\,0)$ and $(0,\, \pm \varepsilon)$, where $\varepsilon = (\sqrt{2}/2)^{k+1}$.

\begin{figure}[ht!]
\centering
\includegraphics[width=0.5\textwidth]{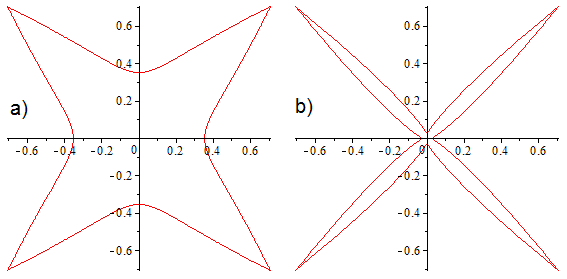}
\caption{Domains $\Omega_{\varepsilon}^*$ under $\varepsilon=\frac{1}{2\sqrt{2}}$ and  $\varepsilon=\frac{1}{32}$.}
\end{figure}

We calculate the partial derivatives of mapping $w$
$$
w_z=(\frac{k}{2}+1)|z|^{k} \quad \text{and} \quad w_{\overline{z}}=\frac{k}{2}|z|^{k-2}z^2.
$$ 
Thus
$$
|Dw|=(k+1)|z|^k \quad \text{and} \quad J(z,w)=(k+1)|z|^{2k}.
$$

Then by Theorem~\ref{T3.5} we have
\begin{multline*}
\frac{1}{\mu_p^{(1)}(\Omega_{\varepsilon}^*)} \leq \inf_{q \in I} (k+1)^p
\left(\frac{2}{\pi^d}\left(\frac{1-d}{1/2-d}\right)^{1-d}\right)^p \\
{} \times
\left(\int\limits_{-\frac{\sqrt{2}}{2}}^{\frac{\sqrt{2}}{2}} \left(\int\limits_{-\frac{\sqrt{2}}{2}}^{\frac{\sqrt{2}}{2}} \left(x^2+y^2\right)^{\frac{k \alpha}{2}}~dy\right)~dx \right)^{\frac{2}{\alpha}} \\
{} \times
\left(\int\limits_{-\frac{\sqrt{2}}{2}}^{\frac{\sqrt{2}}{2}} \left(\int\limits_{-\frac{\sqrt{2}}{2}}^{\frac{\sqrt{2}}{2}} \left(x^2+y^2\right)^{\frac{(p-2)kq}{2(p-q)}}~dy\right)~dx \right)^{\frac{p-q}{q}},
\end{multline*}
where $I=[1,2p/(4K-(2K-1)p))$ and $d=1/q-(\alpha -2) \alpha p$.

\vskip 0.3cm




\vskip 0.3cm

Department of Mathematics, Ben-Gurion University of the Negev, P.O.Box 653, Beer Sheva, 8410501, Israel 
 
\emph{E-mail address:} \email{vladimir@math.bgu.ac.il} \\           
       
 Department of Higher Mathematics and Mathematical Physics, Tomsk Polytechnic University, 634050 Tomsk, Lenin Ave. 30, Russia;
 Department of General Mathematics, Tomsk State University, 634050 Tomsk, Lenin Ave. 36, Russia

 \emph{Current address:} Department of Mathematics, Ben-Gurion University of the Negev, P.O.Box 653, 
  Beer Sheva, 8410501, Israel  
							
 \emph{E-mail address:} \email{vpchelintsev@vtomske.ru}   \\
			  
	Department of Mathematics, Ben-Gurion University of the Negev, P.O.Box 653, Beer Sheva, 8410501, Israel 
							
	\emph{E-mail address:} \email{ukhlov@math.bgu.ac.il

\end{document}